\documentclass[12pt,a4paper,reqno,twoside]{amsart}

\usepackage[english]{babel}
\usepackage{stmaryrd}
\usepackage{dsfont}
\usepackage[symbol*,ragged]{footmisc}
\usepackage[colorlinks,linkcolor=red,anchorcolor=blue,citecolor=blue,urlcolor=blue]{hyperref}
\usepackage{color,xcolor}

\usepackage{geometry}
\usepackage{amssymb}
\usepackage{amsmath}
\usepackage{mathrsfs}
\usepackage{amsfonts}
\usepackage{epsfig}

\usepackage{amsthm}
\usepackage{amsxtra}
\usepackage{bbding}
\usepackage{epsfig}
\usepackage{graphicx}
\usepackage{latexsym}
\usepackage{mathbbol}
\usepackage{bbold}

\usepackage{pifont}
\usepackage{wasysym}
\usepackage{skull}
\usepackage{float}

\DeclareSymbolFontAlphabet{\mathbb}{AMSb}
\DeclareSymbolFontAlphabet{\mathbbol}{bbold}

\usepackage{amscd}
\usepackage[all]{xy}
\allowdisplaybreaks[4]
\usepackage{setspace}

\theoremstyle{plain}
\newtheorem{theorem}{\normalfont\scshape Theorem}[section]
\newtheorem{proposition}{\normalfont\scshape Proposition}[section]
\newtheorem{lemma}[proposition]{\normalfont\scshape Lemma}

\newtheorem*{corollary*}{\normalfont\scshape Corollary}
\newtheorem{remark}{\normalfont\scshape Remark}

\theoremstyle{remark}
\newtheorem*{remark*}{\normalfont\scshape Remark}
\newtheorem*{notation}{\normalfont\scshape Notation}

\numberwithin{equation}{section}
\addtocounter{footnote}{1}

\renewcommand{\footnoterule}{
  \kern -3pt
  \hrule width 2.5in height 0.4pt
  \kern 3pt
}

\makeatletter
\@ifundefined{MakeUppercase}{}{}
\makeatother

\begin{document}
	
\title[Asymmetric two-dimensional divisor problem with congruence conditions ]
	  {On the mean square of the error term for the asymmetric two--dimensional divisor problem with
       congruence conditions}

\author[Zhen Guo, Jinjiang Li, Linji Long, Min Zhang]
       {Zhen Guo \quad \& \quad Jinjiang Li \quad \& \quad Linji Long \quad \& \quad Min Zhang}

\address{[Zhen Guo] Zhiyuan School of Liberal Arts, Beijing Institute of Petro--Chemical Technology,
          Beijing 102617, People's Republic of China}

\email{zhen.guo.math@gmail.com}

\address{[Jinjiang Li] (Corresponding author) Department of Mathematics, China University of Mining and Technology,
          Beijing 100083, People's Republic of China}

\email{jinjiang.li.math@gmail.com}

\address{[Linji Long] Department of Mathematics, China University of Mining and Technology,
          Beijing 100083, People's Republic of China}

\email{linji.long.math@gmail.com}

\address{[Min Zhang] School of Applied Science, Beijing Information Science and Technology University,
          Beijing 100192, People's Republic of China}

\email{min.zhang.math@gmail.com}

\date{}

\footnotetext[1]{Jinjiang Li is the corresponding author. 

\quad\,\,
{\textbf{MR(2020) Subject Classification}}: 11N37, 11N60

}

\begin{abstract}
Suppose that $a$ and $b$ are positive integers subject to $(a,b)=1$. For $n\in\mathbb{Z}^+$, denote by $\tau_{a,b}(n;\ell_1,M_1,l_2,M_2)$ the asymmetric two--dimensional divisor function with congruence conditions, i.e.,
\begin{equation*}
\tau_{a,b}(n;\ell_1,M_1,l_2,M_2)=\sum_{\substack{n=n_1^an_2^b\\ n_1\equiv\ell_1\!\!\!\!\!\pmod{M_1}\\
n_2\equiv\ell_2\!\!\!\!\!\pmod{M_2}}}1.
\end{equation*}
In this paper, we shall establish an asymptotic formula of the mean square of the error
term of the sum $\sum_{n\leqslant M_1^aM_2^bx}\tau_{a,b}(n;\ell_1,M_1,l_2,M_2)$. This result constitutes
an enhancement upon the previous result of Zhai and Cao \cite{MR2564393}.
\end{abstract}

\maketitle

\section{\texorpdfstring{Introduction}{}}
For a positive integer $n$, denote by $\tau(n)$ the classical Dirichlet divisor function. For fixed integers $M_1$
and $M_2$, the divisor function with congruence conditions is defined by
\begin{equation*}
  \tau(n;\ell_1,M_1,l_2,M_2):=\#\{(n_1,n_2)\in\mathbb{N}^2:\;{n_1}{n_2}=n,\;n_i\equiv \ell_i\!\!\!\!\pmod{M_i},\,i=1,2\},
\end{equation*}
where $1\leqslant \ell_1\leqslant M_1$ and $1\leqslant \ell_2\leqslant M_2$. Its generating function is
\begin{equation*}
  \sum_{n=1}^{\infty}\frac{\tau(n;\ell_1,M_1,\ell_2,M_2)}{n^s}=\zeta(s,\lambda_1)\zeta(s,\lambda_2)(M_1M_2)^{-s},
\end{equation*}
where $\Re s>1$, $\lambda_i=\frac{\ell_i}{M_i}$, $i=1,2$, and $\zeta(s,\lambda)$ denotes the Hurwitz zeta--function, which is defined by
\begin{equation*}
\zeta(s,\alpha):=\sum_{n=0}^\infty\frac{1}{(n+\alpha)^s},\qquad s=\sigma+it,\qquad \sigma>1,
\end{equation*}
for $0<\alpha\leqslant1$. The \textit{divisor problem with congruence conditions} is to investigate the error term
\begin{equation*}
\Delta(x;\ell_1,M_1,\ell_2,M_2):=\sum_{n\leqslant x}\tau(n;\ell_1,M_1,\ell_2,M_2)
-\sum_{s_0=0,1}\mathop{\text{Res}}_{s=s_0}\bigg(\frac{\zeta(s,\lambda_1)\zeta(s,\lambda_2)}{s}
\bigg(\frac{x}{M_1M_2}\bigg)^s\bigg).
\end{equation*}
It is conjectured that, for any $\varepsilon>0$, the following estimate
\begin{equation*}
\Delta(M_1M_2x;\ell_1,M_1,\ell_2,M_2)\ll x^{1/4+\varepsilon}
\end{equation*}
holds uniformly in $1\leqslant \ell_1\leqslant M_1$ and $1\leqslant \ell_2\leqslant M_2$. The hitherto best result in this direction is
\begin{equation*}
  \Delta(M_1M_2x;\ell_1,M_1,\ell_2,M_2)\ll x^{\frac{131}{416}}(\log x)^{\frac{26947}{8320}},
\end{equation*}
uniformly in $1\leqslant \ell_1\leqslant M_1$ and $1\leqslant \ell_2\leqslant M_2$, which follows from Richert \cite{MR54655} and Huxley \cite{MR2005876}, consecutively and respectively.

In 1995, M\"{u}ller and Nowak \cite{MR1344601} investigated the mean value of $\Delta(M_1M_2x;\ell_1,M_1,\ell_2,M_2)$. They proved that, for a large real parameter $T$, there holds
\begin{equation*}
  \int_{1}^{T}\Delta(M_1M_2x;\ell_1,M_1,\ell_2,M_2)\mathrm{d}x\ll x^{3/4},
\end{equation*}
and, by using Landau's \cite{MR1544390} method, there holds
\begin{equation}\label{landau}
  \int_{1}^{T}\Delta^2(M_1M_2x;\ell_1,M_1,\ell_2,M_2)dx=\mathfrak{C}_2\left(\frac{\ell_1}{M_1},\frac{\ell_2}{M_2}\right)T^{3/2}+o(T^{3/2})
\end{equation}
(even a good error term), uniformly in $1\leqslant \ell_1\leqslant M_1$ and $1\leqslant \ell_2\leqslant M_2$, where $\mathfrak{C}_2(\ell_1/M_1,\ell_2/M_2)$ is a computable constant determined by $\ell_1/M_1$ and $\ell_2/M_2$.

Suppose $a$ and $b$ are two fixed integers which satisfy $1\leqslant a\leqslant b$. Without loss of generality, we postulate that $(a,b)=1$, and define $\tau_{a,b}(n):=\sum_{n=n_1^an_2^b}1$. The \textit{asymmetric two--dimensional divisor problem} is to study the error term
\begin{equation*}
  \Delta_{a,b}(x):=\sum_{n\leqslant x}\tau_{a,b}(n)-\zeta\left(b/a\right)x^{1/a}-\zeta\left(a/b\right)x^{1/b},
\end{equation*}
under the condition $a\neq b$. If $a=b$, then the appropriate limit should be taken in the above sum. This problem also attracts plenty of interests of many authors.
When $a=b=1$, $\Delta_{1,1}(x)$ is the error term of the classical Dirichlet divisor problem, whose upper bound estimates, $\Omega$--results, sign changes and mean values were investigated by many mathematicians during the past centuries. In this paper, we mainly focus on the case $a\neq b$, so we omit the details, herein, of the results about $\Delta_{1,1}(x)$.

When $a\neq b$, Richert \cite{MR50577} proved that
\begin{equation*}
  \Delta_{a,b}(x)\ll
  \begin{cases}
    x^{\frac{2}{3(a+b)}}, &  b\leqslant2a, \\
    x^{\frac{2}{5a+2b}}, &  b\geqslant2a.
  \end{cases}
\end{equation*}
Sharper upper bound estimates can be found in \cite{MR242776,MR998378,MR72170,MR181609}. In 1988, Hafner \cite{MR932373} proved that
\begin{equation}\label{Hafner}
  \Delta_{a,b}(x)=\Omega_{+}\big(x^{1/2(a+b)}(\log x)^{b/2(a+b)}\log\log x\big)
\end{equation}
and
\begin{equation*}
  \Delta_{a,b}(x)=\Omega\big(x^{1/2(a+b)e^{U(x)}}\big),
\end{equation*}
where
\begin{equation*}
  U(x)=B(\log\log x)^{b/2(a+b)}(\log\log\log x)^{b/2(a+b)-1}
\end{equation*}
for some positive constant $B>0$. However, Nowak \cite{MR1359135} pointed out that the factor $(\log x)^{b/2(a+b)}$ in (\ref{Hafner}) is not correct. K\"{u}hleitner \cite{MR1462841} proved that
\begin{equation*}
  \Delta_{a,b}(x)=\Omega_{+}(x^{1/2(a+b)}(\log x)^{a/2(a+b)}(\log\log x)^{1+(2\log2-1)a/2(a+b)}\text{exp}(-c\sqrt{\log\log\log x})),
\end{equation*}
which is the sharpest result of this kind up to date.

It is conjectured that the estimate
\begin{equation*}
  \Delta_{a,b}(x)=O(x^{1/2(a+b)+\varepsilon})
\end{equation*}
holds for any $1\leqslant a\leqslant b$ with $(a,b)=1$. When $a\neq b$, by using the theory of the Riemann zeta--fuction, Ivi\'{c} \cite{MR0904009}  proved that
\begin{equation*}
  \int_{1}^{T}\Delta^2_{a,b}(x)\mathrm{d}x
  \begin{cases}
    \ll T^{1+1/(a+b)}\log^2T, \\
    =\Omega(T^{1+1/(a+b)}),
  \end{cases}
\end{equation*}
who also conjectured that the asymptotic formula
\begin{equation}\label{Ivic}
  \int_{1}^{T}\Delta^2_{a,b}(x)\mathrm{d}x=\mathfrak{c}_{a,b}T^{1+1/(a+b)}(1+o(1))
\end{equation}
holds for some positive constant $\mathfrak{c}_{a,b}$ in the same literature. In 2010, Zhai and Cao \cite{MR2564393} proved this conjecture and gave a more precise result, i.e.,
\begin{equation}\label{Zhai Cao}
  \int_{1}^{T}\Delta^2_{a,b}(x)\mathrm{d}x=\mathfrak{c}_{a,b}T^{\frac{1+a+b}{a+b}}+O(T^{\frac{1+a+b}{a+b}-\frac{a}{2b(a+b)(a+b-1)}}\log^{7/2}T),
\end{equation}
where
\begin{align*}
  \mathfrak{c}_{a,b}:= &\,\,\frac{a^{b/(a+b)}b^{a/(a+b)}}{2(a+b+1)\pi^2}\sum_{n=1}^{\infty}g_{a,b}^2(n),
\end{align*}
\begin{align*}
  g_{a,b}(n):=&\,\,\sum_{n=h^ar^b}h^{-\frac{a+2b}{2a+2b}}r^{-\frac{b+2a}{2a+2b}}.
\end{align*}
The convergence of the infinite series $\sum_{n=1}^{\infty}g_{a,b}^2(n)$ is demonstrated in Section 4 of Zhai and Cao \cite{MR2564393}.

Based on the above results, it is reasonable to consider \textit{asymmetric two--dimensional divisor function with congruence conditions}, i.e.,
\begin{equation*}
\tau_{a,b}(n;\ell_1,M_1,\ell_2,M_2)
:=\#\big\{(n_1,n_2)\in\mathbb{N}^2:\, n_1^an_2^b=n,\;n_i\equiv \ell_i\!\!\!\!\!\pmod{M_i},\,i=1,2\big\},
\end{equation*}
whose generating function is
\begin{equation*}
\sum_{n=1}^{\infty}\frac{\tau_{a,b}(n;\ell_1,M_1,\ell_2,M_2)}{n^s}
=\zeta(as,\lambda_1)\zeta(bs,\lambda_2)( M_1^aM_2^b)^{-s},
\end{equation*}
where $\Re s>1$, $\lambda_i=\frac{\ell_i}{M_i}$, $i=1,2$, and $\zeta(s,\lambda)$ denotes the Hurwitz zeta--function. The \textit{asymmetric two--dimensional divisor problem with congruence conditions} is to investigate the error term
\begin{equation*}
\begin{aligned}
  \Delta_{a,b}(x;\ell_1,M_1,\ell_2,M_2):=\,\,&\sum_{n\leqslant x}\tau_{a,b}(n;\ell_1,M_1,\ell_2,M_2)\\
  &\,\,-\sum_{s_0={0,1}}\mathop{\text{Res}}_{s=s_0}\left(\frac{\zeta(as,\lambda_1)\zeta(bs,\lambda_2)}{s}\left(\frac{x}{ M_1^aM_2^b}\right)^s\right).
\end{aligned}
\end{equation*}

In this paper, we shall study the mean square moment of $\Delta_{a,b}(x;\ell_1,M_1,\ell_2,M_2)$ and establish an asymptotic formula.
\begin{theorem}\label{Theorem}
  Suppose that $a$ and $b$  are two fixed integers which satisfy $1\leqslant a<b$ and $(a,b)=1$. For fixed integers $M_1$ and $M_2$, we have
\begin{equation*}
  \int_{1}^{T}\Delta^2_{a,b}( M_1^aM_2^bx;\ell_1,M_1,\ell_2,M_2)\mathrm{d}x=\mathfrak{c}^{*}_{a,b}T^{\frac{1+a+b}{a+b}}+O(T^{\frac{1+a+b}{a+b}-\frac{a}{2b(a+b)(a+b-1)}}\log^{7/2}T),
\end{equation*}
where
\begin{equation*}
  \mathfrak{c}^{*}_{a,b}=\frac{a^{b/(a+b)}b^{a/(a+b)}}{2(a+b+1)\pi^2}\sum_{n=1}^{\infty}{g^{*}_{a,b}}(n),
\end{equation*}
and
\begin{equation*}
  {g^{*}_{a,b}}(n)=\sum_{n=h_1^ar_1^b=h_2^ar_2^b}(h_1h_2)^{-\frac{a+2b}{2a+2b}}(r_1r_2)^{-\frac{b+2a}{2a+2b}}
  \cos\left(2\pi\left(\frac{(r_2-r_1)\ell_2}{M_2}+\frac{(h_2-h_1)\ell_1}{M_1}\right)\right).
\end{equation*}
\end{theorem}

\begin{remark}\label{conv of g*}
  It is easy to see that the function ${g^{*}_{a,b}}(n)$ is symmetric for $a$ and $\,b$, i.e., ${g^{*}_{a,b}}(n)={g^{*}_{b,a}}(n)$. The convergence of the infinite series $\sum_{n=1}^{\infty}{g^{*}_{a,b}}(n)$ can be obtained by the convergence of $\sum_{n=1}^{\infty}g_{a,b}^2(n)$, since $|{g^{*}_{a,b}}(n)|\leqslant g_{a,b}^2(n)$.
\end{remark}
\begin{remark}
  Theorem \ref{Theorem} also holds for $a=b=1$. In this case, we have
\begin{equation*}
  \mathfrak{c}^{*}_{1,1}=\sum_{n=1}^{\infty}\frac{1}{n^{3/2}}\sum_{n=h_1^ar_1^b=h_2^ar_2^b}\cos\left(2\pi\left(\frac{(r_2-r_1)\ell_2}{M_2}+\frac{(h_2-h_1)\ell_1
  }{M_1}\right)\right)
  =\mathfrak{C}_2\left(\frac{\ell_1
  }{M_1},\frac{\ell_2}{M_2}\right).
\end{equation*}
Hence our theorem provides a new approach of the proof of (\ref{landau}).
\end{remark}

\begin{notation}
Throughout this paper, $\varepsilon$ denotes a sufficiently small positive number, not necessarily the same at each occurrence. $\mathbb{Z}$ denotes the set of all integers. We use $[t],\{t\}$ and $\|t\|$ to denote the integral part of $t$, the fractional part of $t$ and the distance from $t$ to the nearest integer, respectively. We write $\psi(t)=t-[t]-1/2,e(x)=e^{2\pi ix}$. As usual, denote by $\mu(n)$  the M\"{o}bious' function. $(m,n)$ denotes the greatest common divisor of natural numbers $m$ and $n$. The notation $n\sim N$ means $N<n\leqslant2N$. Finally, we define
\begin{equation*}
  \sideset{}{'}\sum_{\alpha\leqslant n\leqslant\beta}f(n)=
\begin{cases}
  \sum_{\alpha<n<\beta}f(n),& \alpha\notin\mathbb{Z},\beta\notin\mathbb{Z},\\
  f(\alpha)/2+\sum_{\alpha<n<\beta}f(n),& \alpha\in\mathbb{Z},\beta\notin\mathbb{Z},\\
  f(\alpha)/2+\sum_{\alpha<n<\beta}f(n)+f(\beta)/2,& \alpha\in\mathbb{Z},\beta\in\mathbb{Z}.
\end{cases}
\end{equation*}
\end{notation}

\section{Priliminary lemmas}
In this section, we shall list some lemmas which are necessary for establishing Theorem \ref{Theorem}.
\begin{lemma}\label{Heath Brown}
  Let $H\geqslant2$ be any real number. Then
\begin{equation*}
  \psi(u)=-\sum_{1\leqslant|h|\leqslant H}\frac{e(hu)}{2\pi ih}+O\left(\min\left(1,\frac{1}{H\|u\|}\right)\right).
\end{equation*}
\end{lemma}
\begin{proof}
See the arguments on page 245 of Heath--Brown \cite{MR698168}.
\end{proof}

\begin{lemma}\label{Vinogradov}
  Suppose that $f(x)$ and $\varphi(x)$ are algebraic functions, which satisfy the following conditions constrained on the interval $[a,b]$:
\begin{equation*}
\begin{aligned}
  &|f''(x)|\asymp R^{-1},\hspace{1cm}|f'''(x)|\ll(RU)^{-1},\hspace{1cm}U\geqslant1,\\
  &|\varphi(x)|\ll H,\hspace{1cm}|\varphi'(x)|\ll HU_1^{-1},\hspace{1cm}U_1\geqslant1.
\end{aligned}
\end{equation*}
Then we have
\begin{align*}
  \sum_{a<n\leqslant b}\varphi(x)e(f(n))=&\,\,\sum_{\alpha<\nu\leqslant\beta}b_{\nu}\frac{\varphi(n_{\nu})}{\sqrt{|f''(n_{\nu})|}}e\big(f(n_{\nu})-\nu n_{\nu}+1/8\big)\\
  &\,\,+O\big(H\log(\beta-\alpha+2)+H(b-a+R)(U^{-1}+U_1^{-1})\big)\\
  &\,\,+O(H\min(\sqrt{R},\max(\langle\alpha\rangle^{-1},\langle\beta\rangle^{-1}))),
\end{align*}
where $[\alpha,\beta]$ is the image of $[a,b]$ under the mapping $y=f'(x)$; $n_{\nu}$ is the solution of the equation $f'(x)=\nu$;
\begin{equation*}
  b_{\nu}=
\begin{cases}
  1/2, & \mbox{if $\nu=\alpha\in\mathbb{Z}$ or $\nu=\beta\in\mathbb{Z}$},\\
  \,\,\,1, & \mbox{if $\alpha<\nu<\beta$};
\end{cases}
\end{equation*}
and the function $\langle\cdot\rangle$ is defined by
\begin{equation*}
  \langle t\rangle=
\begin{cases}
  \,\,\,\|t\|, & \mbox{if $t\notin\mathbb{Z}$}, \\
  \beta-\alpha, & \mbox{otherwise},
\end{cases}
\end{equation*}
where $\|t\|=\min_{m\in\mathbb{Z}}|t-m|$.
\end{lemma}
\begin{proof}
  See Theorem 1 of Chapter III of Karatsuba and Voronin \cite{Karatsuba-Voronin-1992}.
\end{proof}

\begin{lemma}\label{conv of g}
  Let $g_{a,b}(n)$ be defined as in (\ref{Zhai Cao}). Then for sufficiently large real parameter $y>1$, there holds
\begin{equation*}
  \sum_{n>y}g^2_{a,b}(n)\ll y^{-\frac{a}{(a+b)b}}.
\end{equation*}
\end{lemma}
\begin{proof}
  See Section 4 of Zhai and Cao \cite{MR2564393}.
\end{proof}

\begin{lemma}\label{Sab T est}
  Suppose that $a$ and $b$ are fixed positive integers, and $T$ is a large parameter. Define
\begin{align*}
  S_{a,b}(T):=&\,\,\mathop{\sum_{h_1^ar_1^b\leqslant T^{100(a+b)}}\qquad\sum_{h_2^ar_2^b\leqslant T^{100(a+b)}}}_{0<\left|(h_1^{a}r_1^{b})^{\frac{1}{a+b}}-(h_2^{a}r_2^{b})^{\frac{1}{a+b}}\right|
  <\frac{1}{10}(h_1h_2)^{\frac{a}{2a+2b}}(r_1r_2)^{\frac{b}{2a+2b}}}
  (h_1h_2)^{-\frac{a+2b}{2(a+b)}}(r_1r_2)^{-\frac{2a+b}{2(a+b)}}\\
  &\,\,\qquad\qquad\times\min\Bigg(T^{\frac{1}{a+b}},\frac{1}{\left|(h_1h_2)^{\frac{a}{2a+2b}}(r_1r_2)^{\frac{b}{2a+2b}}\right|}\Bigg),
\end{align*}
Then we have
\begin{equation*}
  S_{a,b}(T)\ll T^{\frac{1}{a+b}-\frac{1}{(a+b)(a+b-1)}\times\frac{\min(a,b)}{\max(a,b)}}\log^7T.
\end{equation*}
\end{lemma}
\begin{proof}
  See Lemma 5.2 of Zhai and Cao \cite{MR2564393}.
\end{proof}

\section{\texorpdfstring{An analogue of Vorono\"{\i}'s formula for $\Delta_{a,b}(x;\ell_1,M_1,\ell_2,M_2)$}{}}

For Dirichlet divisor problem, there holds a truncated Vorono\"{\i}'s formula
\begin{equation*}
  \Delta(x)=\frac{x^{1/4}}{\sqrt{2}\pi}\sum_{n\leqslant N}\frac{\tau(n)}{n^{3/4}}\cos\bigg(4\pi\sqrt{nx}-\frac{\pi}{4}\bigg)+O\big(x^{1/2+\varepsilon}N^{-1/2}\big)
\end{equation*}
for $1<N\ll x$. It plays an essential role in the investigation of the higher--power moments of $\Delta(x)$. For $\Delta_{a,b}( M_1^aM_2^bx; \ell_1
,M_1,\ell_2,M_2)$, there is not such a convenient formula at hand, but we can use the finite expression of $\psi(u)$, e.g., Lemma \ref{Heath Brown}, and van der Corput's B--process, i.e., Lemma \ref{Vinogradov}, to derive an analogue formula (e.g., (\ref{Delta ab chaifenhou}) below).

\subsection{\texorpdfstring{The $\psi$--expression of $\Delta_{a,b}( M_1^aM_2^bx;\ell_1,M_1,\ell_2,M_2)$}{}}
\indent
\newline
Suppose that $T$ is a large real number and $x$ satisfies $T\leqslant x\leqslant2T$. By using Dirichlet's
classical argument, we write
\begin{equation*}
  \sum_{n\leqslant M_1^aM_2^bx}\tau_{a,b}(n;\ell_1,M_1,\ell_2,M_2)=\sum_{\substack{ n_1^an_2^b\leqslant M_1^aM_2^bx\\n_1\equiv \ell_1\!\!\!\!\!\pmod {M_1}\\n_2\equiv\ell_2 \!\!\!\!\!\pmod{M_2}}}1=\Sigma_1+\Sigma_2-\Sigma_3,
\end{equation*}
where
\begin{equation*}
\begin{aligned}
  \Sigma_1&=\sum_{\substack{{n_1}^{a+b}\leqslant M_1^aM_2^bx\\n_1\equiv \ell_1\!\!\!\!\!\pmod {M_1}}}
  \sum_{\substack{{n_2}^b\leqslant\frac{ M_1^aM_2^bx}{{n_1}^a}\\n_2\equiv\ell_2\!\!\!\!\!\pmod{M_2}}}1,
          \hspace{2cm}
  \Sigma_2=\sum_{\substack{{n_2}^{a+b}\leqslant M_1^aM_2^bx\\n_2\equiv \ell_2\!\!\!\!\!\pmod {M_2}}}
  \sum_{\substack{{n_1}^a\leqslant\frac{ M_1^aM_2^bx}{{n_2}^b}\\n_1\equiv\ell_1\!\!\!\!\!\pmod {M_1}}}1,\\
  \Sigma_3&=\sum_{\substack{{n_1}^{a+b}\leqslant M_1^aM_2^bx\\n_1\equiv \ell_1\!\!\!\!\!\pmod {M_1}}}
  \sum_{\substack{{n_2}^{a+b}\leqslant M_1^aM_2^bx\\n_2\equiv \ell_2\!\!\!\!\!\pmod {M_2}}}1.
\end{aligned}
\end{equation*}
Computing the above three sums directly, it is easy to derive that
\begin{equation}\label{Delta ab chaifen}
  \Delta_{a,b}( M_1^aM_2^bx; \ell_1,M_1,\ell_2,M_2)=F_{12}(x)+F_{21}(x)+O(1),
\end{equation}
where
\begin{equation}\label{def F12}
  F_{12}(x):=-\sum_{\substack{{n_1}^{a+b}\leqslant M_1^aM_2^bx\\n_1\equiv \ell_1\!\!\!\!\!\pmod {M_1}}}
  \psi\bigg(\bigg(\bigg(\frac{M_1}{n_1}\bigg)^ax\bigg)^{\frac{1}{b}}-\frac{\ell_2}{M_2}\bigg),
\end{equation}
\begin{equation}\label{def F21}
  F_{21}(x):=-\sum_{\substack{{n_2}^{a+b}\leqslant M_1^aM_2^bx\\n_2\equiv \ell_2\!\!\!\!\!\pmod {M_2}}}
  \psi\bigg(\bigg(\bigg(\frac{M_2}{n_2}\bigg)^bx\bigg)^{\frac{1}{a}}-\frac{\ell_1}{M_1}\bigg).
\end{equation}
\begin{remark}
  In fact $F_{12}(x)$ and $F_{21}(x)$ depend on $a,b,\ell_1,M_1,\ell_2,M_2$, although this is not demonstrated explicitly. Similar situations happen to the following notations such as $R_{12}(x;H)$, $G_{12}(x;H)$, etc.
\end{remark}
Suppose that $H$ is a parameter subject to $T^{\varepsilon}\ll H\ll T^{100(a+b)}$, which is to be determined later. By Lemma \ref{Heath Brown}, one has
\begin{equation}\label{F12 chaifen}
  F_{12}(x)=R_{12}(x;H)+G_{12}(x;H),
\end{equation}
where
\begin{equation}\label{R12 def}
  R_{12}(x;H):=\frac{1}{2\pi i}\sum_{1\leqslant|h|\leqslant H}\frac{1}{h}
  \sum_{\substack{{n_1}^{a+b}\leqslant M_1^aM_2^bx\\n_1\equiv \ell_1
  \!\!\!\!\!\pmod {M_1}}}
  e\bigg(h\bigg(\bigg(\frac{M_1}{n_1}\bigg)^ax\bigg)^{\frac{1}{b}}-\frac{h\ell_2}{M_2}\bigg),
\end{equation}
\begin{equation}\label{G12 def}
  G_{12}(x;H):=O\Bigg(\sum_{\substack{{n_1}^{a+b}\leqslant M_1^aM_2^bx\\n_1\equiv \ell_1\!\!\!\!\!\pmod {M_1}}}
  \min\Bigg(1,\frac{1}{H\big\|\big(\big(\frac{M_1}{n_1}\big)^ax\big)^{\frac{1}{b}}-\frac{\ell_2}{M_2}\big\|}\Bigg)\Bigg).
\end{equation}
Define
\begin{equation*}
\begin{aligned}
  c:=&\,\,(2ab)^{ab},\hspace{0.5cm}  J:=\,\,(\mathscr{L}/(a+b)-\log\mathscr{L})\log^{-1}c,  \hspace{0.5cm}\mathscr{L}:=\log T,\\
  m_j:=&\,\,( M_1^aM_2^bx)^{\frac{1}{a+b}}c^{-j}\qquad(j\geqslant0).
\end{aligned}
\end{equation*}
It is easy to see that
\begin{equation*}
  c^{J}\asymp T^{\frac{1}{a+b}}\mathscr{L}^{-1}.
\end{equation*}
By a splitting argument, we get
\begin{equation}\label{R12 chaifen}
  R_{12}(x;H)=-\frac{\Sigma_{12}}{2\pi i}+\frac{\overline{\Sigma_{12}}}{2\pi i}+O(\mathscr{L}^2),
\end{equation}
where
\begin{equation}\label{Sigma 12 def}
  \Sigma_{12}=\sum_{1\leqslant h\leqslant H}\frac{1}{h}\sum_{j=0}^{J}S(x,h,j),
\end{equation}
\begin{equation}\label{S shj def}
  S(x,h,j):=\sum_{\substack{m_{j+1}<n_1\leqslant m_j\\n_1\equiv \ell_1
  \!\!\!\!\!\pmod {M_1}}}
  e\bigg(-h\bigg(\bigg(\frac{M_1}{n_1}\bigg)^ax\bigg)^{\frac{1}{b}}+\frac{h\ell_2}{M_2}\bigg).
\end{equation}
Define
\begin{equation*}
\begin{aligned}
  c_1(a,b):=&\,\,a^{\frac{b}{2(a+b)}}b^{\frac{a}{2(a+b)}}(a+b)^{-\frac{1}{2}},\\
  c_2(a,b):=&\,\,\left(\frac{a}{b}\right)^{\frac{b}{a+b}}+\left(\frac{b}{a}\right)^{\frac{a}{a+b}}.
\end{aligned}
\end{equation*}
It is easy to check that
\begin{equation*}
  c_1(a,b)=c_1(b,a),\hspace{2cm}c_2(a,b)=c_2(b,a).
\end{equation*}
Applying Lemma \ref{Vinogradov} to (\ref{S shj def}), we get
\begin{align}\label{S(s,h,j)-expan}
           S(s,h,j)
  = &\,\,c_1(a,b)x^{\frac{1}{2(a+b)}}{\sideset{}{'}\sum_{n_{j,h}(a,b)\leqslant r\leqslant
                n_{j+1,h}(a,b)}}h^{\frac{b}{2(a+b)}}r^{-\frac{a+2b}{2(a+b)}}
                     \nonumber \\
  &\,\,\times e\bigg(-c_2(a,b)x^{\frac{1}{a+b}}(h^br^a)^{\frac{1}{a+b}}+
       \frac{h\ell_2}{M_2}+\frac{r\ell_1}{M_1}-\frac{1}{8}\bigg)+O(\mathscr{L}),
\end{align}
where
\begin{equation*}
  n_{j,h}(a,b):=n_{j,h}(a,b;M_1,M_2)=\frac{a}{b}h(2ab)^{a(a+b)j}\frac{M_1}{M_2}.
\end{equation*}
From (\ref{Sigma 12 def}) and (\ref{S(s,h,j)-expan}), we deduce that
\begin{align}\label{Sigma 12 fenduan}
  \Sigma_{12}=&\,\,c_1(a,b)x^{\frac{1}{2(a+b)}}\sum_{1\leqslant h\leqslant H}
  \sum_{\frac{a}{b}h\frac{M_1}{M_2}\leqslant r\leqslant n_{J+1,h}(a,b)}h^{-\frac{2a+b}{2(a+b)}}r^{-\frac{a+2b}{2(a+b)}}\notag\\
  &\,\,\times e\left(-c_2(a,b)x^{\frac{1}{a+b}}(h^br^a)^{\frac{1}{a+b}}+\frac{h\ell_2}{M_2}+\frac{r\ell_1}{M_1}-\frac{1}{8}\right)+O(\mathscr{L}^3).
\end{align}
Combining (\ref{R12 chaifen}) and (\ref{Sigma 12 fenduan}), one obtains
\begin{equation*}
  R_{12}(x;H)=R^{\,*}_{12}(a,b;x)+O(\mathscr{L}^3),
\end{equation*}
where
\begin{equation*}
\begin{aligned}
  R^{\,*}_{12}(a,b;x):=&\,\,\frac{c_1(a,b)}{\pi}x^{\frac{1}{2(a+b)}}\sum_{1\leqslant h\leqslant H}
  {\sideset{}{'}\sum_{\frac{a}{b}h\frac{M_1}{M_2}\leqslant r\leqslant n_{J+1,h}(a,b)}}h^{-\frac{2a+b}{2(a+b)}}r^{-\frac{a+2b}{2(a+b)}}\\
  &\,\,\times\cos\left(2\pi c_2(a,b)(h^br^a)^{\frac{1}{a+b}}x^{\frac{1}{a+b}}-2\pi\left(\frac{h\ell_2}{M_2}+\frac{r\ell_1
  }{M_1}+\frac{1}{8}\right)\right).
\end{aligned}
\end{equation*}
Define
\begin{equation*}
\begin{aligned}
   g(a,b;n,H,J):=&\,\,{\sideset{}{'}\sum_{\substack{n=h^br^a\\1\leqslant h\leqslant H\\ \frac{aM_1}{bM_2}h\leqslant r\leqslant n_{J+1,h}(a,b)}}}h^{-\frac{2a+b}{2(a+b)}}r^{-\frac{a+2b}{2(a+b)}},\\
   g(a,b;n):=&\,\,{\sideset{}{'}\sum_{\substack{n=h^br^a\\ \frac{aM_1}{bM_2}h\leqslant r}}}h^{-\frac{2a+b}{2(a+b)}}r^{-\frac{a+2b}{2(a+b)}}.
\end{aligned}
\end{equation*}
It is easy to check that if $h^br^a\leqslant\min(H^{a+b},T^{a/b})\mathscr{L}^{-a-a^2/b-1}$, $\frac{aM_1}{bM_2}h\leqslant r$, then it follows $h\leqslant H$ and $r\leqslant n_{J+1,h}(a,b)$, and thus
\begin{equation}\label{g ab qudiao HJ}
  g(a,b;n,H,J)=g(a,b;n),\hspace{1cm}n\leqslant\min(H^{a+b},T^{a/b})\mathscr{L}^{-a-a^2/b-1}.
\end{equation}
Therefore, we derive that
\begin{equation*}
\begin{aligned}
  {R_{12}}^{*}(a,b;x)=&\,\,\frac{c_1(a,b)}{\pi}x^{\frac{1}{2(a+b)}}\sum_{1\leqslant n\leqslant H^b(n_{J+1,H}(a,b))^a}g(a,b;n,H,J)\\
  &\,\,\times\cos\left(2\pi c_2(a,b)x^{\frac{1}{a+b}}n^{\frac{1}{a+b}}-2\pi\left(\frac{h\ell_2}{M_2}+\frac{r\ell_1
  }{M_1}+\frac{1}{8}\right)\right).
\end{aligned}
\end{equation*}
For $F_{21}(x)$, we have
\begin{equation*}
  F_{21}(x)=R_{21}(x;H)+G_{21}(x;H),
\end{equation*}
where
\begin{equation}\label{R21 def}
  R_{21}(x;H):=\frac{1}{2\pi i}\sum_{1\leqslant|h|\leqslant H}\frac{1}{h}\sum_{\substack{{n_2}^{a+b}\leqslant  M_1^aM_2^bx\\n_2\equiv \ell_2\!\!\!\!\!\pmod {M_2}}}
  e\bigg(h\bigg(\bigg(\frac{M_2}{n_2}\bigg)^bx\bigg)^{\frac{1}{a}}-\frac{h\ell_1}{M_1}\bigg),
\end{equation}
\begin{equation}\label{G21 def}
  G_{21}(x;H):=O\Bigg(\sum_{\substack{{n_2}^{a+b}\leqslant M_1^aM_2^bx\\n_2\equiv \ell_2\!\!\!\!\!\pmod {M_2}}}
  \min\Bigg(1,\frac{1}{H\big\|\big(\big(\frac{M_2}{n_2}\big)^bx\big)^{\frac{1}{a}}
  -\frac{\ell_1}{M_1}\big\|}\Bigg)\Bigg).
\end{equation}
Similarly, we have
\begin{equation*}
  R_{21}(x;H)=R_{21}^{\,*}(b,a;x)+O(\mathscr{L}^3),
\end{equation*}
where
\begin{equation*}
\begin{aligned}
  R_{21}^{\,*}(b,a;x)=&\,\,\frac{c_1(b,a)}{\pi}x^{\frac{1}{2(a+b)}}\sum_{1\leqslant n\leqslant H^a(n_{J+1,H}(b,a))^b}g(b,a;n,H,J)\\
  &\,\,\times\cos\left(2\pi c_2(b,a)x^{\frac{1}{a+b}}n^{\frac{1}{a+b}}-2\pi\left(\frac{h\ell_1}{M_1}+\frac{r\ell_2}{M_2}+\frac{1}{8}\right)\right),\\
  g(b,a;n,H,J):=&\,\,{\sideset{}{'}\sum_{\substack{n=h^ar^b\\1\leqslant h\leqslant H\\ \frac{bM_2}{aM_1}h\leqslant r\leqslant n_{J+1,h}(b,a)}}}h^{-\frac{a+2b}{2(a+b)}}r^{-\frac{2a+b}{2(a+b)}},\\
  g(b,a;n):=&\,\,{\sideset{}{'}\sum_{\substack{n=h^ar^b\\ \frac{bM_2}{aM_1}h\leqslant r}}}h^{-\frac{a+2b}{2(a+b)}}r^{-\frac{2a+b}{2(a+b)}}.
\end{aligned}
\end{equation*}
It is easy to check that if $h^ar^b\leqslant\min(H^{a+b},T^{b/a})\mathscr{L}^{-b-b^2/a-1}$, $\frac{bM_2}{aM_1}h\leqslant r$, then it follows $h\leqslant H$ and $r\leqslant n_{J+1,h}(b,a)$, and thus
\begin{equation}\label{g ba qudiao HJ}
  g(b,a;n,H,J)=g(b,a;n)\hspace{1cm}n\leqslant\min(H^{a+b},T^{b/a})\mathscr{L}^{-b-b^2/a-1}.
\end{equation}
Suppose that $z$ is a parameter subject to  $T^{\varepsilon}\ll z\leqslant\min(H^{a+b},T^{a/b})\mathscr{L}^{-a-a^2/b-1}$. Define
\begin{equation*}
\begin{aligned}
   R_{12,1}^{\,*}(a,b;x):=&\,\,\frac{c_1(a,b)}{\pi}x^{\frac{1}{2(a+b)}}\sum_{1\leqslant n\leqslant z}g(a,b;n)\\
   &\,\,\times\cos\left(2\pi c_2(a,b)x^{\frac{1}{a+b}}n^{\frac{1}{a+b}}-2\pi\left(\frac{h\ell_2}{M_2}+\frac{r\ell_1}{M_1}+\frac{1}{8}\right)\right),\\
   R_{12,2}^{\,*}(a,b;x):=&\,\,R_{12}^{\,*}(a,b;x)-R_{12,1}^{\,*}(a,b;x),\\
   R_{21,1}^{\,*}(b,a;x):=&\,\,\frac{c_1(b,a)}{\pi}x^{\frac{1}{2(a+b)}}\sum_{1\leqslant n\leqslant z}g(b,a;n)\\
   &\,\,\times\cos\left(2\pi c_2(b,a)x^{\frac{1}{a+b}}n^{\frac{1}{a+b}}-2\pi\left(\frac{h\ell_1
   }{M_1}+\frac{r\ell_2}{M_2}+\frac{1}{8}\right)\right),\\
   R_{21,2}^{\,*}(b,a;x):=&\,\,R_{21}^{\,*}(b,a;x)-R_{21,1}^{\,*}(b,a;x).
\end{aligned}
\end{equation*}
By the definition of $g(a,b;n)$ and $g(b,a;n)$, and exchanging the order of $h$ and $r$ in $R_{21,1}^{\,*}(b,a;x)$, we find that
\begin{equation*}
\begin{aligned}
  &\,\,\sum_{1\leqslant n\leqslant z}g(b,a;n)\cos\left(2\pi c_2(b,a)x^{\frac{1}{a+b}}n^{\frac{1}{a+b}}-2\pi\left(\frac{h\ell_1}{M_1}+\frac{r\ell_2}{M_2}+\frac{1}{8}\right)\right)\\
  =&\,\,\sideset{}{'}\sum_{1\leqslant n\leqslant z}{\sum_{\substack{n=h^ar^b\\ \frac{bM_2}{aM_1}h\leqslant r}}}h^{-\frac{a+2b}{2(a+b)}}r^{-\frac{2a+b}{2(a+b)}}
  \cos\left(2\pi c_2(a,b)x^{\frac{1}{a+b}}n^{\frac{1}{a+b}}-2\pi\left(\frac{h\ell_1}{M_1}+\frac{r\ell_2}{M_2}+\frac{1}{8}\right)\right),
\end{aligned}
\end{equation*}
and
\begin{equation*}
\begin{aligned}
  &\,\,\sum_{1\leqslant n\leqslant z}g(a,b;n)\cos\left(2\pi c_2(a,b)x^{\frac{1}{a+b}}n^{\frac{1}{a+b}}-2\pi\left(\frac{h\ell_2}{M_2}+\frac{r\ell_1}{M_1}+\frac{1}{8}\right)\right)\\
  =&\,\,\sideset{}{'}\sum_{1\leqslant n\leqslant z}{\sum_{\substack{n=h^br^a\\ \frac{aM_1}{bM_2}h\leqslant r}}}h^{-\frac{2a+b}{2(a+b)}}r^{-\frac{a+2b}{2(a+b)}}
  \cos\left(2\pi c_2(a,b)x^{\frac{1}{a+b}}n^{\frac{1}{a+b}}-2\pi\left(\frac{h\ell_2}{M_2}+\frac{r\ell_1}{M_1}+\frac{1}{8}\right)\right)\\
  =&\,\,\sideset{}{'}\sum_{1\leqslant n\leqslant z}{\sum_{\substack{n=h^ar^b\\ \frac{aM_1}{bM_2}r\leqslant h}}}h^{-\frac{a+2b}{2(a+b)}}r^{-\frac{2a+b}{2(a+b)}}
  \cos\left(2\pi c_2(a,b)
  x^{\frac{1}{a+b}}n^{\frac{1}{a+b}}-2\pi\left(\frac{h\ell_1}{M_1}+\frac{r\ell_2}{M_2}+\frac{1}{8}\right)\right),
\end{aligned}
\end{equation*}
where we use $c_2(a,b)=c_2(b,a)$. Thus, we have
\begin{equation*}
\begin{aligned}
   &\,\,R_{12,1}^{\,*}(a,b;x)+R_{21,1}^{\,*}(b,a;x)\\
   =&\,\,\frac{c_1(a,b)}{\pi}x^{\frac{1}{2(a+b)}}\sum_{1\leqslant n\leqslant z}\sum_{n=h^ar^b}h^{-\frac{a+2b}{2(a+b)}}r^{-\frac{2a+b}{2(a+b)}}
   \cos\left(2\pi c_2(a,b)x^{\frac{1}{a+b}}n^{\frac{1}{a+b}}-2\pi\left(\frac{h\ell_1}{M_1}+\frac{r\ell_2}{M_2}+\frac{1}{8}\right)\right).
\end{aligned}
\end{equation*}
Define
\begin{equation*}
\begin{aligned}
  \Delta^{*}_{a,b}(x;z):=&\,\,\frac{c_1(a,b)}{\pi}x^{\frac{1}{2(a+b)}}\sum_{1\leqslant n\leqslant z}\sum_{n=h^ar^b}h^{-\frac{a+2b}{2(a+b)}}r^{-\frac{2a+b}{2(a+b)}}\\
  &\,\,\times\cos\left(2\pi c_2(a,b)x^{\frac{1}{a+b}}n^{\frac{1}{a+b}}-2\pi\left(\frac{h\ell_1}{M_1}+\frac{r\ell_2}{M_2}+\frac{1}{8}\right)\right).
\end{aligned}
\end{equation*}
Based on the above arguments, we obtain
\begin{align}\label{Delta ab chaifenhou}
  &\,\,\Delta_{a,b}( M_1^aM_2^bx; \ell_1,M_1,\ell_2,M_2)=\Delta^{*}_{a,b}(x;z)+E_{a,b}(x),\nonumber\\
  &\,\,E_{a,b}(x):=-R_{12,2}^{\,*}(a,b;x)-R_{21,2}^{\,*}(b,a;x)+G_{12}(x;H)+G_{21}(x;H)+O(\mathscr{L}^3).
\end{align}
The formula (\ref{Delta ab chaifenhou}) can be regarded as a truncated Vorono\"{\i}'s formula.

\section{Proof of Theorem \ref{Theorem}}
In this section, we shall finish the proof of Theorem \ref{Theorem}. It suffices for us to evaluate the integral
\begin{equation*}
 \int_{T}^{2T}\Delta^2_{a,b}(x;\ell_1,M_1,\ell_2,M_2)\mathrm{d}x,
\end{equation*}
where $T\geqslant10$ is a large parameter.
\subsection{\texorpdfstring{Mean square of $\Delta^{*}_{a,b}(x;z)$}{}}
\indent
\newline
Suppose that $H$ and $z$ are parameters, which satisfy
\begin{equation*}
  T^{\varepsilon}\ll H\ll T^{100(a+b)},\qquad T^{\varepsilon}\ll z\leqslant\min\left(H^{a+b},T^{\frac{a}{b}}\right)\mathscr{L}^{-b-\frac{b^2}{a}-1}.
\end{equation*}
By the elementary formula
\begin{equation*}
  \cos{u}\cos{v}=\frac{1}{2}(\cos(u-v)+\cos(u+v)),
\end{equation*}
we may write
\begin{align}\label{S 123 def}
  |\Delta^{*}_{a,b}(x;z)|^2=&\,\,\frac{c_1^2(a,b)}{\pi^2}x^{\frac{1}{a+b}}\sum_{1\leqslant n_1,n_2\leqslant z}\sum_{n_1=h_1^ar_1^b}h_1^{-\frac{a+2b}{2(a+b)}}r_1^{-\frac{2a+b}{2(a+b)}}
  \sum_{n_2=h_2^ar_2^b}h_2^{-\frac{a+2b}{2(a+b)}}r_2^{-\frac{2a+b}{2(a+b)}}\nonumber\\
  &\,\,\times\cos\left(-2\pi\left(\frac{r_1\ell_2}{M_2}+\frac{h_1\ell_1}{M_1}+\frac{1}{8}\right)+2\pi\left(\frac{r_2\ell_2}{M_2}+\frac{h_2\ell_1}{M_1}+\frac{1}{8}\right)\right)\nonumber\\
  =&\,\,S_1(x)+S_2(x)+S_3(x),
\end{align}
where
\begin{align*}
          S_1(x)
 = & \,\, \frac{c_1^2(a,b)}{2\pi^2}x^{\frac{1}{a+b}}\sum_{1\leqslant n\leqslant z}
          \sum_{n=h_1^ar_1^b=h_2^ar_2^b}h_1^{-\frac{a+2b}{2(a+b)}}r_1^{-\frac{2a+b}{2(a+b)}}
          h_2^{-\frac{a+2b}{2(a+b)}}r_2^{-\frac{2a+b}{2(a+b)}}
                \nonumber \\
   & \,\, \times\cos\left(2\pi\left(\frac{(r_2-r_1)\ell_2}{M_2}+\frac{(h_2-h_1)\ell_1}{M_1}\right)\right),
                \nonumber \\
          S_2(x)
   =&\,\, \frac{c_1^2(a,b)}{2\pi^2}x^{\frac{1}{a+b}}\sum_{\substack{1\leqslant n_1,n_2\leqslant z\\n_1\neq n_2}}
          \sum_{n_1=h_1^ar_1^b}h_1^{-\frac{a+2b}{2(a+b)}}r_1^{-\frac{2a+b}{2(a+b)}}
          \sum_{n_2=h_2^ar_2^b}h_2^{-\frac{a+2b}{2(a+b)}}r_2^{-\frac{2a+b}{2(a+b)}}
                \nonumber \\
    & \,\, \times\cos\left(2\pi c_2(a,b)x^{\frac{1}{a+b}}\left(n_1^{\frac{1}{a+b}}-n_2^{\frac{1}{a+b}}\right)
           +2\pi\left(\frac{(r_2-r_1)\ell_2}{M_2}+\frac{(h_2-h_1)\ell_1}{M_1}\right)\right),
                \nonumber \\
           S_3(x)
  = & \,\, \frac{c_1^2(a,b)}{2\pi^2}x^{\frac{1}{a+b}}\sum_{1\leqslant n_1,n_2\leqslant z}
           \sum_{n_1=h_1^ar_1^b}h_1^{-\frac{a+2b}{2(a+b)}}r_1^{-\frac{2a+b}{2(a+b)}}
           \sum_{n_2=h_2^ar_2^b}h_2^{-\frac{a+2b}{2(a+b)}}r_2^{-\frac{2a+b}{2(a+b)}}
                \nonumber \\
    & \,\, \times\sin\left(2\pi c_2(a,b)x^{\frac{1}{a+b}}
            \left(n_1^{\frac{1}{a+b}}+n_2^{\frac{1}{a+b}}\right)
            -2\pi\left(\frac{(r_1+r_2)\ell_2}{M_2}+\frac{(h_1+h_2)\ell_1}{M_1}\right)\right).
\end{align*}
It follows from Remark \ref{conv of g*} and Lemma \ref{conv of g} that
\begin{align}\label{S1 int eva}
\int_{T}^{2T}S_1(x)\mathrm{d}x=&\,\,\frac{c_1^2(a,b)}{2\pi^2}\sum_{n=1}^{\infty}{g^{*}_{a,b}}(n)
\int_{T}^{2T}x^{\frac{1}{a+b}}\mathrm{d}x+O\left(\sum_{n>z}{g^2_{a,b}}(n)\int_{T}^{2T}x^{\frac{1}{a+b}}\mathrm{d}x\right)\nonumber\\
  =&\,\,\frac{c_1^2(a,b)}{2\pi^2}\sum_{n=1}^{\infty}{g^{*}_{a,b}}(n)\int_{T}^{2T}x^{\frac{1}{a+b}}\mathrm{d}x
  +O\Big(T^{\frac{1+a+b}{a+b}}z^{-\frac{a}{(a+b)b}}\Big).
\end{align}
By the first derivative estimate, and changing the order of integration and summation, we derive that
\begin{equation*}
\begin{aligned}
  \int_{T}^{2T}S_3(x)\mathrm{d}x\ll&\,\,\sum_{1\leqslant n_1,n_2\leqslant z}g_{a,b}(n_1)g_{a,b}(n_2)\frac{T}{{n_1}^{\frac{1}{a+b}}+{n_2}^{\frac{1}{a+b}}}\\
  \ll&\,\,T\sum_{1\leqslant n_1,n_2\leqslant z}g_{a,b}(n_1)g_{a,b}(n_2)\frac{1}{(n_1n_2)^{\frac{1}{2(a+b)}}}\\
  \ll&\,\,T\left(\sum_{n\leqslant z}\frac{g_{a,b}(n)}{n^{\frac{1}{2(a+b)}}}\right)^2,
\end{aligned}
\end{equation*}
where $g_{a,b}(n)$ is defined by (\ref{Zhai Cao}), and in the second step we used the well--known inequality $\alpha^2+\beta^2\geqslant2\alpha\beta$. By Euler's product, we know that, for $\Re s>1$, there holds
\begin{equation*}
  \sum_{n=1}^{\infty}\frac{g_{a,b}(n)}{n^s}=\zeta\left(as+\frac{a+2b}{2a+2b}\right)\zeta\left(bs+\frac{2a+b}{2a+2b}\right),
\end{equation*}
which can be continued meromorphically to the whole complex plane with a double pole at $s=1/2(a+b)$. Thus, we use  Perron's formula to deduce that
\begin{equation*}
  \sum_{n\leqslant X}g_{a,b}(n)\ll X\log X,
\end{equation*}
which implies that
\begin{equation}\label{g ab(n) par sum}
  \sum_{n\leqslant X}g_{a,b}(n)n^{-\frac{1}{2(a+b)}}\ll\log^2X.
\end{equation}
From the above arguments, we get
\begin{equation}\label{S3 int est}
  \int_{T}^{2T}S_3(x)\mathrm{d}x\ll T\mathscr{L}^4.
\end{equation}
Now we consider the contribution of $S_2(x)$. Write
\begin{equation}\label{S2 chaifen}
  S_2(x)=S_{21}(x)-S_{22}(x),
\end{equation}
where
\begin{align*}
         S_{21}(x)
= & \,\, \frac{c_1^2(a,b)}{2\pi^2}x^{\frac{1}{a+b}}\mathop{\sum_{n_1,n_2\leqslant z}
         \sum_{n_1=h_1^ar_1^b}\sum_{n_2=h_2^ar_2^b}}_{\big|n_1^{\frac{1}{a+b}}-n_2^{\frac{1}{a+b}}\big|\geqslant
         \frac{(n_1n_2)^{\frac{1}{2(a+b)}}}{10}}h_1^{-\frac{a+2b}{2(a+b)}}r_1^{-\frac{2a+b}{2(a+b)}}
         h_2^{-\frac{a+2b}{2(a+b)}}r_2^{-\frac{2a+b}{2(a+b)}}
               \nonumber \\
  & \,\, \times\cos\left(2\pi c_2(a,b)x^{\frac{1}{a+b}}
         \Big(n_1^{\frac{1}{a+b}}-n_2^{\frac{1}{a+b}}\Big)+2\pi\left(\frac{(r_2-r_1)\ell_2}{M_2}+
         \frac{(h_2-h_1)\ell_1}{M_1}\right)\right),
                \nonumber \\
         S_{22}(x)
= & \,\, \frac{c_1^2(a,b)}{2\pi^2}x^{\frac{1}{a+b}}\mathop{\sum_{n_1,n_2\leqslant z}
         \sum_{n_1=h_1^ar_1^b}\sum_{n_2=h_2^ar_2^b}}_{\big|n_1^{\frac{1}{a+b}}-n_2^{\frac{1}{a+b}}\big|
         <\frac{(n_1n_2)^{\frac{1}{2(a+b)}}}{10}}h_1^{-\frac{a+2b}{2(a+b)}}r_2^{-\frac{2a+b}{2(a+b)}}
         h_2^{-\frac{a+2b}{2(a+b)}}r_2^{-\frac{2a+b}{2(a+b)}}
                 \nonumber \\
  & \,\, \times\cos\left(2\pi c_2(a,b)
         x^{\frac{1}{a+b}}\Big(n_1^{\frac{1}{a+b}}-n_2^{\frac{1}{a+b}}\Big)
         +2\pi\left(\frac{(r_2-r_1)\ell_2}{M_2}+\frac{(h_2-h_1)\ell_1}{M_1}\right)\right).
\end{align*}
Similar to the case $S_3(x)$, we have
\begin{equation*}
  \int_{T}^{2T}S_{21}(x)\mathrm{d}x\ll T\mathscr{L}^4.
\end{equation*}
By first derivative test and Lemma \ref{Sab T est}, we get
\begin{equation*}
\begin{aligned}
  \int_{T}^{2T}S_{22}(x)\mathrm{d}x\ll&\,\,T\sum_{\substack{n_1,n_2\leqslant z\\\big|n_1^{\frac{1}{a+b}}-n_2^{\frac{1}{a+b}}\big|<\frac{(n_1n_2)^{\frac{1}{2(a+b)}}}{10}}}
  g_{a,b}(n_1)g_{a,b}(n_2)\min\Big(T^{\frac{1}{a+b}},\big|n_1^{\frac{1}{a+b}}-{n_2}^{\frac{1}{a+b}}\big|^{-1}\Big)\\
  \ll&\,\,TS_{a,b}(T)\ll T^{1+\frac{1}{a+b}-\frac{a}{b(a+b)(a+b-1)}}\mathscr{L}^7.
\end{aligned}
\end{equation*}
From the above two estimates, we get
\begin{equation}\label{S2 int est}
  \int_{T}^{2T}S_2(x)\mathrm{d}x\ll T^{1+\frac{1}{a+b}-\frac{a}{b(a+b)(a+b-1)}}\mathscr{L}^7.
\end{equation}
From (\ref{S 123 def}), (\ref{S1 int eva}), (\ref{S3 int est}) and (\ref{S2 int est}), one has
\begin{align}\label{Delta* ab int eva}
  \int_{T}^{2T}|\Delta^{*}_{a,b}(x;z)|^2\mathrm{d}x=&\,\,\frac{c_1^2(a,b)}{2\pi^2}\sum_{n=1}^{\infty}{g^{*}_{a,b}}(n)\int_{T}^{2T}x^{\frac{1}{a+b}}\mathrm{d}x\nonumber\\
  &\,\,+O\left(T^{\frac{1+a+b}{a+b}}z^{-\frac{a}{(a+b)b}}+T^{\frac{1+a+b}{a+b}-\frac{a}{b(a+b)(a+b-1)}}\mathscr{L}^7\right).
\end{align}

\subsection{\texorpdfstring{Mean squares of ${R_{12,2}}^{*}(a,b;x)$ and ${R_{21,2}}^{*}(b,a;x)$}{}}
\indent
\newline
In this subsection, we shall investigate the mean squares of ${R_{12,2}}^{*}(a,b;x)$ and ${R_{21,2}}^{*}(b,a;x)$. Recall that
\begin{equation*}
\begin{aligned}
  R_{12,2}^{\,*}(a,b;x)=&\,\,\frac{c_1(a,b)}{\pi}x^{\frac{1}{2(a+b)}}\sum_{z<n\leqslant H^b(n_{J+1,H}(a,b))^a}g(a,b;n,H,J)\\
  &\,\,\times\cos\left(2\pi c_2(a,b)x^{\frac{1}{a+b}}n^{\frac{1}{a+b}}-2\pi\left(\frac{h\ell_2}{M_2}+\frac{r\ell_1}{M_1}+\frac{1}{8}\right)\right).
\end{aligned}
\end{equation*}
Thus, we have
\begin{align}\label{S4 5 def}
  |R_{12,2}^{\,*}(a,b;x)|^2\ll&\,\,x^{\frac{1}{a+b}}\sum_{z<n_1,n_2\leqslant H^b(n_{J+1,H}(a,b))^a}g(a,b;n_1,H,J)g(a,b;n_2,H,J)\notag\\
  &\,\,\times e\left(c_2(a,b)x^{\frac{1}{a+b}}\left({n_1}^{\frac{1}{a+b}}-{n_2}^{\frac{1}{a+b}}\right)+\frac{(h_2-h_1)\ell_2}{M_2}+\frac{(r_2-r_1)\ell_1}{M_1}\right)\notag\\
  =&\,\,S_4(x)+S_5(x),
\end{align}
where
\begin{equation*}
\begin{aligned}
  S_4(x):=&\,\,x^{\frac{1}{a+b}}\sum_{z<n\leqslant H^b(n_{J+1,H}(a,b))^a}g^2(a,b;n,H,J)e\left(\frac{(h_2-h_1)\ell_2}{M_2}+\frac{(r_2-r_1)\ell_1}{M_1}\right),\\
  S_5(x):=&\,\,x^{\frac{1}{a+b}}\sum_{\substack{z<n_1,n_2\leqslant H^b(n_{J+1,H}(a,b))^a\\n_1\neq n_2}}g(a,b;n_1,H,J)g(a,b;n_2,H,J)\\
  &\,\,\times e\left(c_2(a,b)x^{\frac{1}{a+b}}\left({n_1}^{\frac{1}{a+b}}-{n_2}^{\frac{1}{a+b}}\right)+\frac{(h_2-h_1)\ell_2}{M_2}+\frac{(r_2-r_1)\ell_1}{M_1}\right).
\end{aligned}
\end{equation*}
It is easy to see that
\begin{equation*}
  \left|g^2(a,b;n,H,J)e\left(\frac{(h_2-h_1)\ell_2}{M_2}+\frac{(r_2-r_1)\ell_1}{M_1}\right)\right|\leqslant g^2(a,b;n,H,J)\leqslant g_{a,b}^2(n),
\end{equation*}
which combined with Lemma \ref{conv of g} yields
\begin{equation}\label{S4 int est}
  \int_{T}^{2T}S_4(x)\mathrm{d}x\ll \sum_{n>z}g_{a,b}^2(n)\int_{T}^{2T}x^{\frac{1}{a+b}}\mathrm{d}x\ll T^{1+\frac{1}{a+b}}z^{-\frac{a}{b(a+b)}}.
\end{equation}
By the first derivative test and the inequality $g(a,b;n,H,J)\leqslant g_{a,b}(n)$, we get
\begin{align}\label{int S5 chaifen}
  \int_{T}^{2T}S_5(x)\mathrm{d}x\ll&\,\,T\sum_{\substack{z\leqslant n_1,n_2\leqslant H^bT^a\\n_1\neq n_2}}g_{a,b}(n_1)g_{a,b}(n_2)
  \min\Big(T^{\frac{1}{a+b}},\big|n_1^{\frac{1}{a+b}}-n_2^{\frac{1}{a+b}}\big|^{-1}\Big)\nonumber\\
  =& \,\,T(\Sigma_4+\Sigma_5),
\end{align}
where
\begin{equation*}
\begin{aligned}
  \Sigma_4:=&\,\,\sum_{\substack{z<n_1,n_2\leqslant H^bT^a\\ 0<\big|n_1^{\frac{1}{a+b}}-n_2^{\frac{1}{a+b}}\big|\leqslant\frac{1}{10}(n_1n_2)^{\frac{1}{2(a+b)}}}}
  g_{a,b}(n_1)g_{a,b}(n_2)\min\Big(T^{\frac{1}{a+b}},\big|n_1^{\frac{1}{a+b}}-n_2^{\frac{1}{a+b}}\big|^{-1}\Big),\\
  \Sigma_5:=&\,\,\sum_{\substack{z<n_1,n_2\leqslant H^bT^a\\ \big|n_1^{\frac{1}{a+b}}-n_2^{\frac{1}{a+b}}\big|>\frac{1}{10}(n_1n_2)^{\frac{1}{2(a+b)}}}}
  g_{a,b}(n_1)g_{a,b}(n_2)\min\Big(T^{\frac{1}{a+b}},\big|{n_1}^{\frac{1}{a+b}}-{n_2}^{\frac{1}{a+b}}\big|^{-1}\Big).
\end{aligned}
\end{equation*}
By Lemma \ref{Sab T est}, we have
\begin{equation}\label{Sigma4 est}
  \Sigma_4\ll S_{a,b}(T)\ll T^{1+\frac{1}{a+b}-\frac{a}{b(a+b)(a+b-1)}}\mathscr{L}^7.
\end{equation}
For $\Sigma_5$, it follows from (\ref{g ab(n) par sum}) that
\begin{equation}\label{Sigma5 est}
  \Sigma_5\ll\Bigg(\sum_{n\leqslant H^bT^a}g_{a,b}(n)n^{-\frac{1}{2(a+b)}}\Bigg)^2\ll\mathscr{L}^4.
\end{equation}
From (\ref{S4 int est})--(\ref{Sigma5 est}), we deduce that
\begin{equation}\label{R12,2* int est}
  \int_{T}^{2T}|R_{12,2}^{\,*}(a,b;x)|^2\mathrm{d}x\ll T^{1+\frac{1}{a+b}}z^{-\frac{a}{b(a+b)}}+T^{1+\frac{1}{a+b}-\frac{a}{b(a+b)(a+b-1)}}\mathscr{L}^7.
\end{equation}
Similarly, we obtain
\begin{equation}\label{R21,2* int est}
  \int_{T}^{2T}|R_{21,2}^{\,*}(b,a;x)|^2\mathrm{d}x\ll T^{1+\frac{1}{a+b}}z^{-\frac{a}{b(a+b)}}+T^{1+\frac{1}{a+b}-\frac{a}{b(a+b)(a+b-1)}}\mathscr{L}^7.
\end{equation}
\subsection{\texorpdfstring{Mean squares of $G_{12}(x;H)$ and $G_{21}(x;H)$}{}}
\indent
\newline
In this subsection, we shall study the mean squares of $G_{12}(x;H)$ and $G_{21}(x;H)$. Recall that
\begin{equation*}
  G_{12}(x;H)=O\Bigg(\sum_{\substack{{n_1}^{a+b}\leqslant M_1^aM_2^bx\\n_1\equiv \ell_1(\bmod M_1)}}\min\Bigg(1,\frac{1}{H\big\|\big(\big(\frac{M_1}{n_1}\big)^ax\big)^{\frac{1}{b}}-\frac{\ell_2}{M_2}\big\|}\Bigg)\Bigg).
\end{equation*}
Therefore, one has
\begin{align*}
   &\,\,   \int\limits_{T}^{2T}G_{12}(x;H)\mathrm{d}x
  \ll  \sum_{\substack{{n_1}\leqslant(2M_1^{a}M_2^{b}T)^{\frac{1}{a+b}}\\n_1\equiv \ell_1(\bmod M_1)}}
             \int\limits_{T}^{2T}\min\Bigg(1,\frac{1}{H\big\|\big(\big(\frac{M_1}{n_1}\big)^ax\big)^{\frac{1}{b}}
             -\frac{\ell_2}{M_2}\big\|}\Bigg)\mathrm{d}x
                   \nonumber \\
  \ll & \,\, \sum_{\substack{{n_1}\leqslant(2M_1^{a}M_2^{b}T)^{\frac{1}{a+b}}\\n_1\equiv \ell_1(\bmod M_1)}}
             \int\limits_{\left(\left(\frac{M_1}{n_1}\right)^aT\right)^{\frac{1}{b}}
             -\frac{\ell_2}{M_2}}^{\left(2\left(\frac{M_1}{n_1}\right)^aT\right)^{\frac{1}{b}}-\frac{\ell_2}{M_2}}
             \min\left(1,\frac{1}{H\|u\|}\right)\mathrm{d}
             \Bigg(\frac{\big(u+\frac{\ell_2}{M_2}\big)^b}{\big(\frac{M_1}{n_1}\big)^a}\Bigg)
                   \nonumber \\
  \ll & \,\, \sum_{\substack{{n_1}\leqslant(2M_1^{a}M_2^{b}T)^{\frac{1}{a+b}}\\n_1\equiv \ell_1(\bmod M_1)}}
             \int\limits_{\left(\left(\frac{M_1}{n_1}\right)^aT\right)^{\frac{1}{b}}
             -\frac{\ell_2}{M_2}}^{\left(2\left(\frac{M_1}{n_1}\right)^aT\right)^{\frac{1}{b}}-\frac{\ell_2}{M_2}}
             \frac{\big(u+\frac{\ell_2}{M_2}\big)^{b-1}}{\big(\frac{M_1}{n_1}\big)^a}\min\left(1,\frac{1}{H\|u\|}\right)\mathrm{d}u
                   \nonumber \\
  \ll & \,\, \sum_{{n_1}\leqslant(2M_1^{a}M_2^{b}T)^{\frac{1}{a+b}}}
             \bigg(\left(\frac{M_1}{n_1}\right)^aT\bigg)^{\frac{b-1}{b}}\left(\frac{M_1}{n_1}\right)^{-a}
             \int\limits_{\left(\left(\frac{M_1}{n_1}\right)^aT\right)^{\frac{1}{b}}
             -\frac{\ell_2}{M_2}}^{\left(2\left(\frac{M_1}{n_1}\right)^aT\right)^{\frac{1}{b}}-\frac{\ell_2}{M_2}}
             \min\left(1,\frac{1}{H\|u\|}\right)\mathrm{d}u
                  \nonumber \\
  \ll & \,\, T^{\frac{b-1}{b}}\sum_{{n_1}\leqslant(2M_1^{a}M_2^{b}T)^{\frac{1}{a+b}}}
             \left(\frac{M_1}{n_1}\right)^{-\frac{a}{b}}\left(\left(\frac{M_1}{n_1}\right)^aT\right)^{\frac{1}{b}}
             \int_{0}^{1}\min\left(1,\frac{1}{H\|u\|}\right)\mathrm{d}u
                  \nonumber \\
  \ll & \,\, T\sum_{{n_1}\leqslant(2M_1^{a}M_2^{b}T)^{\frac{1}{a+b}}}\int_{0}^{1}
             \min\left(1,\frac{1}{H\|u\|}\right)\mathrm{d}u
             \ll T^{1+\frac{1}{a+b}}\int_{0}^{\frac{1}{2}}\min\left(1,\frac{1}{H\|u\|}\right)\mathrm{d}u
                  \nonumber \\
  \ll & \,\, T^{1+\frac{1}{a+b}}\left(\int_{0}^{\frac{1}{H}}du
             +\int_{\frac{1}{H}}^{\frac{1}{2}}\frac{1}{Hu}\mathrm{d}u\right)
             \ll T^{1+\frac{1}{a+b}}H^{-1}\mathscr{L},
\end{align*}
which combined with the trivial estimate $G_{12}(x;H)\ll T^{\frac{1}{a+b}}$ yields
\begin{equation}\label{G12 int est}
  \int_{T}^{2T}G_{12}^2(x;H)\mathrm{d}x\ll T^{1+\frac{2}{a+b}}H^{-1}\mathscr{L}.
\end{equation}
Similarly, we have
\begin{equation}\label{G21 int est}
  \int_{T}^{2T}G_{21}^2(x;H)\mathrm{d}x\ll T^{1+\frac{2}{a+b}}H^{-1}\mathscr{L}.
\end{equation}
\subsection{Completion of the proof of the Theorem \ref{Theorem}}
\indent
\newline
In this section, we complete the proof of Theorem \ref{Theorem}. We take $H=T^{10(a+b)}$ and $z=T^{a/b}\mathscr{L}^{-b-b^2/a-1}$. From (\ref{Delta ab chaifenhou}), (\ref{R12,2* int est})--(\ref{G21 int est}) we get
\begin{equation}\label{Eab int est}
  \int_{T}^{2T}E_{a,b}^2(x)\mathrm{d}x\ll T^{1+\frac{1}{a+b}-\frac{a}{b(a+b)(a+b-1)}}\mathscr{L}^7,
\end{equation}
which combined with (\ref{Delta* ab int eva}) and Cauchy's inequality yields
\begin{equation}\label{Delta* ab Eab int est}
  \int_{T}^{2T}E_{a,b}(x)\Delta^{*}_{a,b}(x;z)\mathrm{d}x\ll T^{1+\frac{1}{a+b}-\frac{a}{2b(a+b)(a+b-1)}}\mathscr{L}^{\frac{7}{2}}.
\end{equation}
From (\ref{Delta* ab int eva}), (\ref{Eab int est}) and (\ref{Delta* ab Eab int est}), we get
\begin{align}\label{Delta ab int eva}
  \int_{T}^{2T}\Delta_{a,b}^2( M_1^aM_2^bx; \ell_1,M_1,\ell_2,M_2)\mathrm{d}x=&\,\,\frac{c_1^2(a,b)}{2\pi^2}\sum_{n=1}^{\infty}{g^{*}_{a,b}}(n)\int_{T}^{2T}x^{\frac{1}{a+b}}\mathrm{d}x\notag\\
  &\,\,+O(T^{1+\frac{1}{a+b}-\frac{a}{2b(a+b)(a+b-1)}}\mathscr{L}^{\frac{7}{2}}).
\end{align}
From (\ref{Delta ab int eva}) and a splitting argument, we get
\begin{align*}
    & \,\, \int_{1}^{T}\Delta_{a,b}^2( M_1^aM_2^bx; \ell_1,M_1,\ell_2,M_2)\mathrm{d}x
                  \nonumber \\
  = & \,\, \frac{c_1^2(a,b)}{2\pi^2}\sum_{n=1}^{\infty}{g^{*}_{a,b}}(n)
           \int_{T}^{2T}x^{\frac{1}{a+b}}\mathrm{d}x+O\big(T^{1+\frac{1}{a+b}-\frac{a}{2b(a+b)(a+b-1)}}
           \mathscr{L}^{\frac{7}{2}}\big)
                  \nonumber \\
  = &\,\, \frac{c_1^2(a,b)(a+b)}{2(a+b+1)\pi^2}\sum_{n=1}^{\infty}{g^{*}_{a,b}}(n)
          T^{\frac{1+a+b}{a+b}}O\big(T^{1+\frac{1}{a+b}-\frac{a}{2b(a+b)(a+b-1)}}\mathscr{L}^{\frac{7}{2}}\big)
                  \nonumber \\
  = &\,\, \mathfrak{c}^{*}_{a,b}T^{\frac{1+a+b}{a+b}}
          +O\big(T^{\frac{1+a+b}{a+b}-\frac{a}{2b(a+b)(a+b-1)}}\mathscr{L}^{\frac{7}{2}}\big),
\end{align*}
which completes the proof of the Theorem.

\section*{Acknowledgement}

The authors would like to appreciate the referee for his/her patience in refereeing this paper.
This work is supported by Beijing Natural Science Foundation (Grant No. 1242003), and
the National Natural Science Foundation of China (Grant Nos. 12471009, 12301006, 11901566, 12001047).

\end{document}